\documentclass[sn-mathphys,Numbered]{sn-jnl}

\usepackage{graphicx}%
\usepackage{multirow}%
\usepackage{amsmath,amssymb,amsfonts}%
\usepackage{amsthm}%
\usepackage{mathrsfs}%
\usepackage[title]{appendix}%
\usepackage{xcolor}%
\usepackage{textcomp}%
\usepackage{manyfoot}%
\usepackage{booktabs}%
\usepackage{algorithm}%
\usepackage{algorithmicx}%
\usepackage{algpseudocode}%
\usepackage{listings}%
\theoremstyle{thmstyleone}%

\theoremstyle{thmstyletwo}%

\theoremstyle{thmstylethree}%
\newtheorem*{theorem}{Theorem}

\begin{document}

%\title[]{An Easy Proof to Jordan Canonical Form}

\title[]{When Tangent Plane = Limit of Secant Plane}

\author%[1]
{\fnm{Zhibin} \sur{Yan}}\email{zbyan@hit.edu.cn}
\affil%[1]
{\orgdiv{School of Sciences}, \orgname{Harbin Institute of Technology
(Shenzhen)}, \orgaddress{\street{HIT Campus of University Town}, \city{Shenzhen}, \postcode{518055}, \state{Guangdong}, \country{China}}}

\abstract{For function of one variable, differentiability is equivalent to the existence of tangent line as the limit of secant line. The genuine counterpart of this equivalence for function of several variables is obtained for the first time.}


%%\pacs[JEL Classification]{D8, H51}


\maketitle

Firstly, we give an example to illustrate that (total) differentiability (see [1. p. 212]) does not always imply the existence of limit of secant plane for function of two variables. This phenomenon is neither noticed seemingly nor understood well up to now. The meaning of ``limit of secant plane'' will become clear soon, in a sense naturally parallel to one variable case.

Let
$$
z=f(x,y)=x^{2}+y^{2}.
$$
In the $x$-$y$ coordinate plane, 
take points
$$
P=\begin{bmatrix}
          0 \\ 0
        \end{bmatrix},
A_{k}=\begin{bmatrix}
          \sin\frac{1}{k} \\ 0
        \end{bmatrix},
B_{k}=\begin{bmatrix}
          2\sin\frac{1}{k}\cos\frac{1}{k} \\ 2\sin^{2}\frac{1}{k}
        \end{bmatrix},
C_{k}=\begin{bmatrix}
          3\sin\frac{1}{k}\cos\frac{1}{k} \\ 3\sin^{2}\frac{1}{k}
        \end{bmatrix},
$$
where $k=1, 2,\ldots$. Then the equation of the secant plane of the graph of function $f(x, y)$ passing through the three points $(P, f(P)), (A_{k}, f(A_{k}))$ and $(B_{k}, f(B_{k}))$ is
$$
z=\left(\sin\frac{1}{k}\right)x+
\left(2-\cos\frac{1}{k}\right)y.
$$
Corresponding to the other three points $(P, f(P)), (A_{k}, f(A_{k}))$ and $(C_{k}, f(C_{k}))$, the equation is
$$
z=\left(\sin\frac{1}{k}\right)x+
\left(3-\cos\frac{1}{k}\right)y.
$$
When $k\rightarrow \infty$, $A_{k}, B_{k}$ and $C_{k}$ all tend to $P$, and the two secant planes tend to two different planes
$$
z=y
\ \textrm{  and  }\
z=2y
$$
respectively. However, the function is totally differentiable and the tangent plane at $(P, f(P))$ has equation
$$
z=0.
$$

Now we endeavor to understand this phenomenon. Let
$
z=f(x,y)
$
be defined on a neighborhood of point $
P=\begin{bmatrix}
          x_{0} & y_{0}
        \end{bmatrix}^{\textrm{T}}$.
To avoid the superficially complicated limit notations, we discretize the limit process. Take in the neighborhood points
$
A_{k}=\begin{bmatrix}
          x_{k} & y_{k}
        \end{bmatrix}^{\textrm{T}}
$ and
$
B_{k}=\begin{bmatrix}
          u_{k} & v_{k}
        \end{bmatrix}^{\textrm{T}}
$ different from $P$, $k=1, 2,\ldots$,
which satisfy
\begin{equation}\label{EqOne}
\lim_{k\rightarrow\ \infty}A_{k}
=\lim_{k\rightarrow\ \infty}B_{k}=P.
\end{equation}
The angle $\theta_{k}$ between vectors $A_{k}-P$ and $B_{k}-P$ can be calculated by inner product and lengths of vectors as
$$
\theta_{k}=\arccos\frac{(A_{k}-P)\cdot(B_{k}-P)}{|A_{k}-P|\times|B_{k}-P|}.
$$
Denote
$
z_{k}=f(x_{k},y_{k}), w_{k}=f(u_{k},v_{k})
$.
Recall that $f(x,y)$ is called (totally) differentiable at $P$, if there exists a matrix $J$ of size $1\times 2$ such that
\begin{equation}\label{EqTotalDiff}
\lim_{\left| \begin{bmatrix}
          \Delta x & \Delta y
        \end{bmatrix}^{\textrm{T}}\right|\rightarrow 0}\frac{\left|\Delta z-J\begin{bmatrix}
          \Delta x & \Delta y
        \end{bmatrix}^{\textrm{T}}\right|}{\left|\begin{bmatrix}
          \Delta x & \Delta y
        \end{bmatrix}^{\textrm{T}}\right|}=0,
\end{equation}
where
$
\Delta z =f(x_{0}+\Delta x,y_{0}+\Delta y)-f(x_{0},y_{0}).
$
The matrix $J$ is then called the total derivative (Jacobian matrix), distinguished from partial derivatives, two numbers.

\begin{theorem}\label{Theorem}
The following two statements are equivalent:
\begin{enumerate}
  \item Function $f(x,y)$ is totally differentiable at $P$.
  \item For arbitrary point sequences $\{A_{k}\}$ and $\{B_{k}\}$ satisfying (1) and
\begin{equation}\label{EqTheta}
\sin\theta_{k}\geq p, \ \ k=1, 2,\ldots
\end{equation}
where $0<p<1$, the limit (of matrix sequence)
\begin{equation}\label{EqKey}
\lim_{k\rightarrow\infty}
\begin{bmatrix}
z_{k}-z_{0} & w_{k}-z_{0} \\
\end{bmatrix}
\begin{bmatrix}
x_{k}-x_{0} & u_{k}-x_{0} \\
y_{k}-y_{0} & v_{k}-y_{0} \\
\end{bmatrix}^{-1}
\end{equation}
exists.
\end{enumerate}
If the statements hold, the limit (\ref{EqKey}) as an explicitly appeared matrix is the total derivative $J$ as an implicitly defined matrix in (\ref{EqTotalDiff}).
\end{theorem}

\begin{proof}
Statement 1 $\Rightarrow$ Statement 2.
Firstly, (\ref{EqTheta}) implies the existence of the inverse matrix in (\ref{EqKey}). Let $A$ be the known matrix defined in (\ref{EqTotalDiff}). We prove
\begin{eqnarray*}
% \nonumber to remove numbering (before each equation)
   & & \begin{bmatrix}
   z_{k}-z_{0} & w_{k}-z_{0} \\
   \end{bmatrix}
   \begin{bmatrix}
   x_{k}-x_{0} & u_{k}-x_{0} \\
   y_{k}-y_{0} & v_{k}-y_{0} \\
   \end{bmatrix}^{-1}
   -A \\
   &=& \left(
\begin{bmatrix}
   z_{k}-z_{0} & w_{k}-z_{0} \\
   \end{bmatrix}
   -A
   \begin{bmatrix}
   x_{k}-x_{0} & u_{k}-x_{0} \\
   y_{k}-y_{0} & v_{k}-y_{0} \\
   \end{bmatrix}
   \right)
   \begin{bmatrix}
   x_{k}-x_{0} & u_{k}-x_{0} \\
   y_{k}-y_{0} & v_{k}-y_{0} \\
   \end{bmatrix}^{-1}\\
  &=&
\begin{bmatrix}
   z_{k}-z_{0}-A
   \begin{bmatrix}
   x_{k}-x_{0} \\
   y_{k}-y_{0}
   \end{bmatrix} \ & \ w_{k}-z_{0}-A
   \begin{bmatrix}
   u_{k}-x_{0} \\
   v_{k}-y_{0}
   \end{bmatrix}
   \end{bmatrix}
   \begin{bmatrix}
   x_{k}-x_{0} & u_{k}-x_{0} \\
   y_{k}-y_{0} & v_{k}-y_{0} \\
   \end{bmatrix}^{-1}
\end{eqnarray*}
tends to zero matrix. We have
\begin{equation}\label{EqDiag}
 \begin{bmatrix}
   x_{k}-x_{0} & u_{k}-x_{0} \\
   y_{k}-y_{0} & v_{k}-y_{0}
\end{bmatrix}^{-1}
   = \begin{bmatrix}
   |A_{k}-P|^{-1} & 0 \\
   0 & |B_{k}-P|^{-1}
\end{bmatrix}
\begin{bmatrix}
   \frac{A_{k}-P}{|A_{k}-P|} & \frac{B_{k}-P}{|B_{k}-P|}
\end{bmatrix}^{-1}.
\end{equation}
Since
\begin{equation}\label{KeyInverse}
 \begin{bmatrix}
   \frac{A_{k}-P}{|A_{k}-P|} & \frac{B_{k}-P}{|B_{k}-P|}
\end{bmatrix}^{-1}
   =
   \frac{1}{\det\begin{bmatrix}
   \frac{A_{k}-P}{|A_{k}-P|} & \frac{B_{k}-P}{|B_{k}-P|}
\end{bmatrix}}
   \begin{bmatrix}
   \frac{v_{k}-y_{0}}{|B_{k}-P|} & \frac{x_{0}-u_{k}}{|B_{k}-P|} \\
   \frac{y_{0}-y_{k}}{|A_{k}-P|} & \frac{x_{k}-x_{0}}{|A_{k}-P|}
\end{bmatrix}
\end{equation}
and
$$
\left|\det\begin{bmatrix}
   \frac{A_{k}-P}{|A_{k}-P|} & \frac{B_{k}-P}{|B_{k}-P|}
\end{bmatrix}\right|
=\sin\theta_{k},
$$
every entry of the matrix (\ref{KeyInverse}) is bounded by $1/p$ from (3).
Then the result follows from and Eqs. (1), (2) and (\ref{EqDiag}).

Statement 2 $\Rightarrow$ Statement 1.
Firstly, the limit (\ref{EqKey}) is unique (does not depend on the choice of $\{A_{k}\}$ and $\{B_{k}\}$), and we denote it by $A$, i.e.,
\begin{equation}\label{EqKeyWithA}
A=\lim_{k\rightarrow\infty}
\begin{bmatrix}
z_{k}-z_{0} & w_{k}-z_{0} \\
\end{bmatrix}
\begin{bmatrix}
x_{k}-x_{0} & u_{k}-x_{0} \\
y_{k}-y_{0} & v_{k}-y_{0} \\
\end{bmatrix}^{-1}.
\end{equation}
Now for arbitrarily given $\{A_{k}\}$ satisfying $A_{k}\neq P$ and (1), we especially take $\{B_{k}\}$
such that
$$
u_{k}=x_{0}-(y_{k}-y_{0}), \ \ v_{k}=y_{0}+(x_{k}-x_{0}).
$$
Hence
$
B_{k}-P\perp A_{k}-P
$
and
$$ |B_{k}-P|=|A_{k}-P|=\sqrt{(x_{k}-x_{0})^{2}+(y_{k}-y_{0})^{2}}.
$$
Then $\theta_{k}=\pi/2$ and (3) is satisfied.
We only need to prove
\begin{eqnarray*}
% \nonumber to remove numbering (before each equation)
   & & \frac{\begin{bmatrix}
   z_{k}-z_{0} & w_{k}-z_{0} \\
   \end{bmatrix}
   -A
   \begin{bmatrix}
   A_{k}-P & B_{k}-P
   \end{bmatrix}}{|A_{k}-P|}
   \\
   &=& \left(
\begin{bmatrix}
   z_{k}-z_{0} & w_{k}-z_{0}
   \end{bmatrix}
   \begin{bmatrix}
   A_{k}-P & B_{k}-P
   \end{bmatrix}^{-1}-A
   \right)
   \frac{\begin{bmatrix}
   A_{k}-P & B_{k}-P
   \end{bmatrix}}{|A_{k}-P|}
\end{eqnarray*}
tends to zero matrix. This follows from (7) and the fact that the matrix
$$
\frac{\begin{bmatrix}
   A_{k}-P & B_{k}-P
   \end{bmatrix}}{|A_{k}-P|}
=
\frac{\begin{bmatrix}
   x_{k}-x_{0} & y_{0}-y_{k} \\
   y_{k}-y_{0} & x_{k}-x_{0} \\
\end{bmatrix}}
{\sqrt{(x_{k}-x_{0})^{2}+(y_{k}-y_{0})^{2}}}
$$
with its every entry bounded by $1$.
\end{proof}

%\begin{remark}
Denote
$$
\begin{bmatrix}
\alpha_{k} & \beta_{k}
\end{bmatrix}
=
\begin{bmatrix}
z_{k}-z_{0} & w_{k}-z_{0} \\
\end{bmatrix}
\begin{bmatrix}
x_{k}-x_{0} & u_{k}-x_{0} \\
y_{k}-y_{0} & v_{k}-y_{0} \\
\end{bmatrix}^{-1}.
$$
Then the secant plane passing through the three points $(P, f(P), (A_{k}, f(A_{k})$ and $(B_{k}, f(B_{k})$ has the equation
\begin{equation*}
  z=z_{0}+\alpha_{k}(x-x_{0})+\beta_{k}(y-y_{0}).
\end{equation*}
The theorem characterizes the total differentiability directly using the existence of tangent plane as the limit of secant plane. It perfectly parallels the one variable case, except the uniform linear independence condition (3). One such characterization is a longstanding dream in understanding and teaching multivariate calculus.
%\end{remark}
In following, we explain this concisely.

For function $f(x)$ of one variable, if the limit
\begin{equation}
\underset{\Delta x\rightarrow 0}{\lim }\frac{\Delta y}{\Delta x}=\underset{%
\Delta x\rightarrow 0}{\lim }\frac{f(a+\Delta x)-f(a)}{\Delta x}
\label{daoshu}
\end{equation}%
exists, then the function $f(x)$ is called to be derivable at $a$, and the
limit is called the derivative, denoted by $f^{\prime }(a)$. In this way, derivability is defined through an explicit condition on the function $f(x)$ itself; as contrast, the definition (\ref{EqTotalDiff}) of total derivability involves an implicit thing (matrix $J$) beside $f(x,y)$ itself.  

Derivability has the
celebrated interpretations: The slop of the secant line tends to the slop of
the tangent line; the average rate of change tends to the instantaneous rate
of change. Such a way of defining and interpreting derivative trivially applies
to partial derivative for several variables, but partial derivative is essentially derivative of function of one variable. For function of several variables, the genuine counterpart of
derivative is not partial derivative, but total derivative. Consider function $z=f(x,y)$ of two variables, take the fixed point $(x_{0},y_{0})$, and denote $\Delta x=x-x_{0}$, $\Delta y=y-y_{0}$ and $\Delta z=f(x,y)-f(x_{0},y_{0})$. A naive imitating of (%
\ref{daoshu}) demands for some kind of division operation
\begin{equation}
\frac{\Delta z}{\left[
\begin{array}{cc}
\Delta x & \Delta y%
\end{array}%
\right] }\text{ or }\frac{\Delta z}{\left[
\begin{array}{c}
\Delta x \\
\Delta y%
\end{array}%
\right] },  \label{division}
\end{equation}%
which, in the scope of mathematics up to now, has no meaning. 

Now Theorem \ref{Theorem} discovers a clear principle which solves the difficulty: The ``increment" (``quantity of change") of vector (as the mathematical expression of ``several variables") should be matrix
$$
\begin{bmatrix}
   x_{k}-x_{0} & u_{k}-x_{0} \\
   y_{k}-y_{0} & v_{k}-y_{0} \\
\end{bmatrix},
$$
no longer only being a vector
$$
\begin{bmatrix}
   x_{k}-x_{0} \\
   y_{k}-y_{0} \\
\end{bmatrix},
$$
or in geometrical language, the increment of vector should be parallelogram.
Then the algebraic difficulty involved in (\ref{division}) is naturally solved through reformulating it in terms of matrix operations in (\ref{EqKeyWithA}); the expected physical interpretation of ``rate of change" for total derivative survives in a new but natural way for function of several variables.

\bibliography{sn-bibliography}

%%===========================================================================================%%
%% If you are submitting to one of the Nature Portfolio journals, using the eJP submission   %%
%% system, please include the references within the manuscript file itself. You may do this  %%
%% by copying the reference list from your .bbl file, paste it into the main manuscript .tex %%
%% file, and delete the associated \verb+\bibliography+ commands.                            %%
%%===========================================================================================%%

%\bibliography{sn-bibliography}% common bib file
%% if required, the content of .bbl file can be included here once bbl is generated
%%\input sn-article.bbl

\end{document}